\documentclass{amsart}
\usepackage{graphicx}
\usepackage{url}

\newtheorem{theorem}{Theorem}[section]
\newtheorem{lemma}[theorem]{Lemma}
\newtheorem{prop}[theorem]{Proposition}
\newtheorem{cor}[theorem]{Corollary}

\theoremstyle{definition}
\newtheorem{definition}[theorem]{Definition}

\newtheorem*{notation*}{Notation}
\newtheorem{notation}[theorem]{Notation}

\theoremstyle{remark}
\newtheorem{remark}[theorem]{Remark}

\numberwithin{equation}{section}

\newcommand{\abs}[1]{\lvert#1\rvert}

\usepackage{cite}
\nocite*{}

\title{Connected sums of knots do not admit purely cosmetic surgeries}

\author{Ran Tao}
\address{School of mathematical sciences, Sichuan normal university, Chengdu, China}
\email{rtao\_sicnu@163.com}

\date{\today}

\begin{document}

\begin{abstract}
Two Dehn surgeries on a knot are called purely cosmetic if their surgered manifolds are homeomorphic as oriented manifolds. Gordon conjectured  that non-trivial knots in $S^3$ do not admit purely cosmetic surgeries. In this article, we confirm this conjecture for connected sums of knots by analysing the JSJ-structures.

\end{abstract}

\maketitle

\section{introduction} 
\label{sec:introduction}

Let $K$ be a knot in $S^3$, 
we call two different surgery slopes $r$ and $s$ \emph{purely cosmetic} 
if the surgered manifolds $S^3_{r}(K)$ and $S^3_{s}(K)$ are homeomorphic as oriented manifolds.
The Cosmetic Surgery Conjecture (\cite{gordon1990dehn,kirby1995problems}) says that if $K$ is non-trivial, 
then it does not admit purely cosmetic surgeries. 
In \cite{MR3943698}, 
we confirmed this conjecture for cable knots by studying the JSJ-structures. 
It is natural to ask if this method can be applied to other families of knots.
In this article, we show that composite knots, which are connected sums of non-trivial knots,
also satisfy this conjecture. 
Our main result is the following:

\begin{theorem}
\label{main_thm}
Let $J$ be a composite knot. 
Suppose there exists an orientation-preserving homeomorphism $h: S^3_{r}(J) \to S^3_{s}(J)$, then $r = s$.
\end{theorem}

We mention a related result here, which was recently proved in \cite{2019arXiv190606773H} using Heegaard Floer homology. 
Note that our result includes Theorem \ref{hanselman_connected_sum} as partial cases.

\begin{theorem}[Theorem 4 of \cite{2019arXiv190606773H}]
\label{hanselman_connected_sum}
Let $K \subset S^3$ be a non-trivial knot whose prime summands each have at most $16$ crossings, then $S^3_r(K) \not\cong S^3_s(K)$ for $r \neq s$.
\end{theorem}

Our result uses the following obstruction theorem, 
which improves the main result in \cite{ni2015cosmetic}.

\begin{theorem}[\cite{2019arXiv190606773H}]
\label{hanselmans theorem}
If $K$ is a non-trivial knot in $S^3$ and 
$S^3_r(K) \cong S^3_s(K)$ for $r \neq s$,
then we have the following:

\begin{itemize}
	\item The pair of slopes $\{r, s\}$ are either $\{ \pm 2\}$ or $\{ \pm 1/q \}$ for some positive integer $q$;
	\item if $\{r, s\}$ are $\{ \pm 2\}$, then the Seifert genus $g(K) = 2$; 
	\item if $\{r, s\}$ are $\{ \pm 1/q\}$, then $q \le \dfrac{th(K) + 2g(K)}{2g(K)(g(K) - 1)}$, where $th(K)$ is the Heegaard-Floer thickness of $K$.
\end{itemize}
\end{theorem}

The outline of the proof of Theorem \ref{main_thm} is the following: 
first, 
we only need to consider surgery pairs 
$\{ \pm 2\}, \{ \pm 1\}$ and $\{ \pm 1/q\}$ for $q >1$ by Theorem \ref{hanselmans theorem};
if $\{r ,s\} = \{ \pm 1/q \}$ for $q > 2$, 
then the JSJ-pieces of $S^3_r(K)$ and $S^3_s(K)$ are different; 
if $J$ has at least three prime summands and $\{ r, s\} = \{ \pm 1/2\}$, 
or if $J$ has only two prime summands and $\{ r, s\} \subset \mathbb{Z}$, 
then the JSJ-structures are easy to analyse;
the remaining cases are more difficult, 
and we need to analyse more carefully the mappings and symmetries of the surgered manifolds.

\begin{figure}[ht]

\includegraphics[scale=0.2,angle=90]{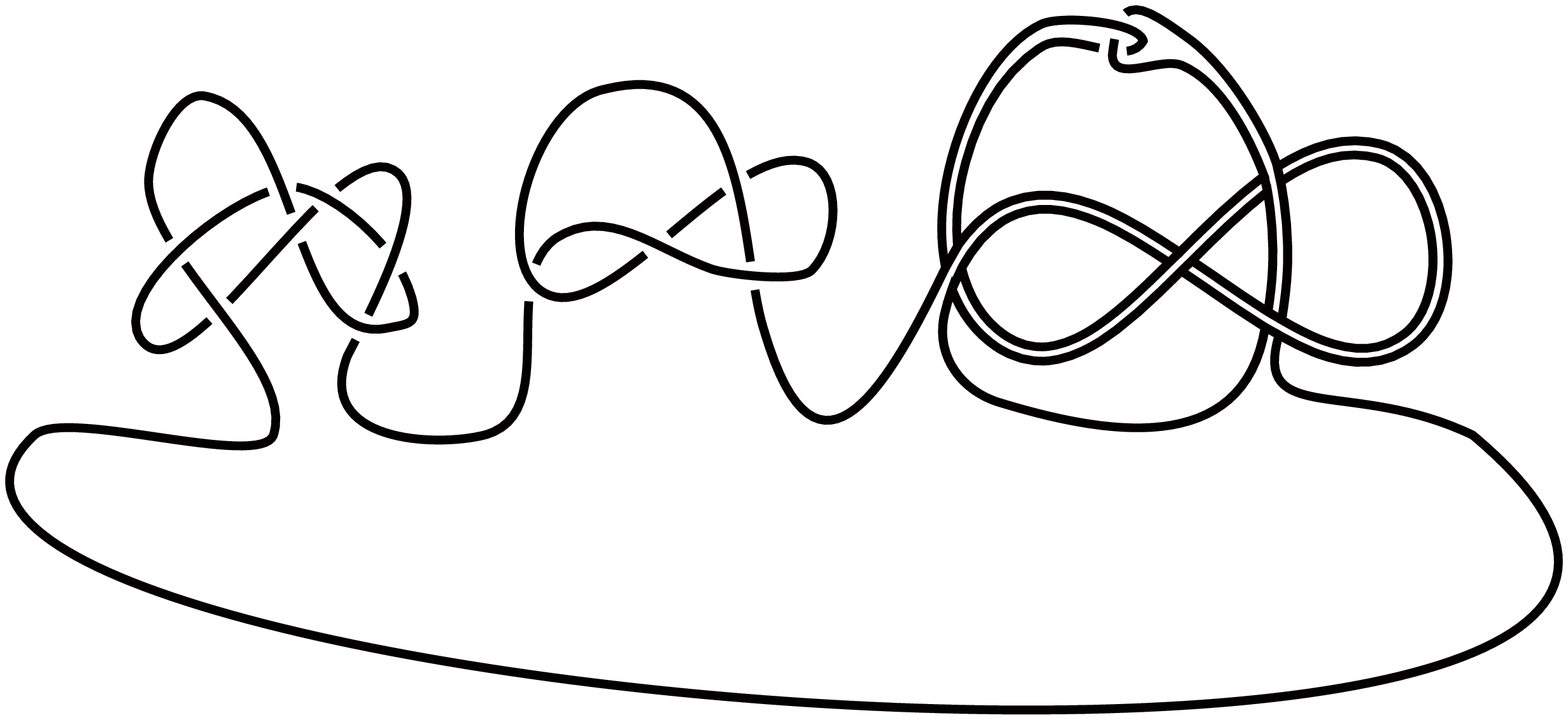}
\includegraphics{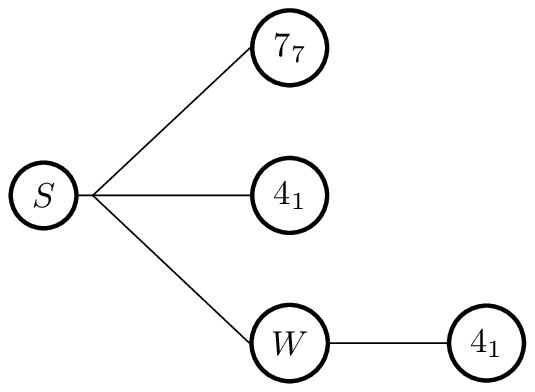}
\caption{Left: A composite knot $J$ with three prime summands: $7_7$, the Figure-$8$ knot, and the Whitehead double of the Figure-$8$ knot.  Right: the corresponding tree structure of the JSJ-decomposition of $E(J)$.
\label{first example}}
\end{figure}

\subsection*{Acknowledgements.} I wish to thank Zhongtao Wu and Jingling Yang for helpful comments and discussions. 
This work was partially supported by Sichuan Science and Technology Program
(No. 2019YJ0509).


\section{The JSJ-decompositions of knot complements}

\begin{notation} 

\label{notation1}

The notation $M \cong M'$ means that there exists an orientation-preserving homeomorphism 
between oriented $3$-manifolds $M$ and $M'$. 
All the $3$-manifolds in this article are assumed to be compact and oriented.

We use $N(K)$ to denote a tubular neighborhood of the knot $K$ in a $3$-manifold and $E(K)$ to denote the complement of $N(K)$ in this manifold. 
We assume $K$ is a knot in  $S^3$ unless otherwise stated. 
For each knot $K$, a prefered longitude is a simple closed curve on $\partial N(K)$ with trivial homology in $E(K)$. 

Fix a  knot $K$, we can talk about the slopes at the boundary torus of $N(K)$ or of $E(K)$. A $(p,q)$-curve or a $p/q$-slope is a simple closed curve that winds $p$ times along the meridional direction and $q$ times along the prefered longitude. In most cases, this torus is the boundary of $N(K)$ (or $E(K)$) of a knot $K$, and the prefered longitude is chosen to be a parallel copy of $K$ on $\partial N(K)$ which has linking number $0$ with $K$. 

We use $C_{p,q}(K)$ to denote the $(p,q)$-cable of $K$, with longitudinal winding number $\abs{q}>1$. Since $C_{p,\pm 1} (K)$ is isotopic to $K,$ we require that $\abs{q} \ge 2$.  

Let $M$ be a $3$-manifold with a toroidal boundary component. Suppose  we can talk about $(p/q)$-slopes on this boundary torus. We use $M(r)$ to denote the resulting manifold of the $r$-slope Dehn filling on $M$.  


We use the standard notation $S(0,n; \alpha/\beta )$ to denote the Seifert fibered space with $n$ boundary components, one singular fiber with coefficient $\alpha/ \beta$, and a base orbifold of genus zero. We omit $\alpha/ \beta$ if there are no singular fibers.

\end{notation}

In this section, we collect some results related to JSJ-decomposition\cite{MR520524,MR551744,hatcher2000notes,budney2005jsj}.
Most materials are taken from \cite{MR3943698}.

\begin{theorem}[The JSJ-decomposition theorem, Theorem 1.9 of \cite{hatcher2000notes}]
Let $M$ be a compact irreducible orientable $3$-manifold. Then there exists a finite collection of embedded tori $\{T_i\}$ in $M$ such that each component of $M\backslash \cup_i T_i$ is either atoroidal or Seifert fibered. Furthermore, a minimal choice of such a collection  is unique up to isotopy. 
\end{theorem}

\begin{definition}
We call the unique isotopy class of decomposition tori (or any representative) in the above theorem the \emph{JSJ-tori} of $M$. 
We call an embedded torus $T$ a \emph{JSJ-torus} if $T$ is isotopic to a torus in the collection of {JSJ-tori}.
We also call the components resulting from decomposing $M$ along the {the JSJ-tori}  \emph{the JSJ-pieces} of $M$. 
We just call an object \emph{JSJ} for short in the above cases if there is no ambiguity.
\end{definition}

\begin{remark}

We can apply the JSJ-decomposition theorem to a manifold with incompressible toroidal boundary by considering its double. 
In particular, knot complements admit JSJ-decompositions. 
See \cite{budney2005jsj} for an explicit description of this JSJ-structure.

\end{remark}


We need a criterion on whether certain tori are JSJ.

\begin{prop}[Proposition 1.6.2 of \cite{aschenbrenner3}]
\label{criterion}

Let $M$ be a compact irreducible orientable $3$-manifold with empty or toroidal boundary. Let $\{T_i\}$ be a collection of disjoint embedded incompressible tori in $M$. Then $\{T_i\}$ are the JSJ-tori of $M$ if and only if the following holds: 
\begin{enumerate}

\item each component $\{M_j\}$ of $M\backslash \cup_i T_i$ is atoroidal or Seifert fibered;

\item if  $~T_i$ cobounds Seifert fibered components $M_j$ and $M_k$ (with possibly $j = k$), then their regular fibers do not match; in other words, their Seifert fibered structures can not be glued together along $T_i$ to form a larger one;

\item if a component $M_i$ is homeomorphic to $T^2 \times I$, then $M$ is a torus bundle with only one JSJ-piece.
\end{enumerate}

\end{prop}



The following lemma says that homeomorphic $3$-manifolds must have identical JSJ-pieces.
This is one of the key ideas of this article.
\begin{lemma}[Lemma 2.4 of \cite{MR3943698}]
\label{diff_sfs_piece}
Let $N$, $F$, and $F'$ be compact $3$-manifolds with toroidal boundaries. Suppose $F$ and $F'$ are atoroidal or Seifert fibered. Let $M = N \cup_T F$ and $M' = N \cup_{T'} F'$ be manifolds obtained by gluing along the boundary tori. Suppose further that the gluing tori $T$ and $T'$ are JSJ in $M$ and $M'$. If $F \not\cong F'$, then $M \not\cong M'.$ 
\end{lemma}

We need a description of the JSJ-decomposition of knot complements in $S^3$. 
The following theorem is due to Budney\cite{budney2005jsj}, 
which is based on previous works by Jaco and
Shalen \cite{MR520524}, 
Johannson\cite{MR551744}, 
Bonahon and Siebenmann\cite{bonahonsiebenmann_unpublish}, 
Eisenbud and Neumann\cite{MR817982}, 
and Thurston \cite{MR648524}.
We use the version reformulated by Lackenby.

\begin{figure}[ht]
\includegraphics{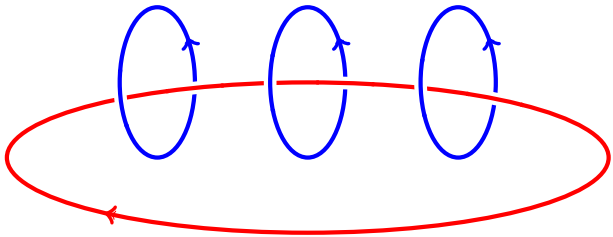}
\caption{ The key-chain link $H_3$. Each `key' corresponds to a prime summand of the composite knot. The distinguished circle corresponds to $\partial E(K)$. The complement of this key-chain link corresponds to the node $S$ in Figure \ref{a_typical_JSJ_tree_of_a_knot}. 
\label{key-chain_3}}
\end{figure}

\begin{theorem}[Theorem 4.1 of \cite{lackenby2017every}, Theorem 4.18 of \cite{budney2005jsj}]
\label{jsj_of_knot_complement}

Suppose $K$ is a knot in $S^3$ such that $S^3 \backslash \text{int}(N(K))$ has at least one JSJ-torus. Let $M$ be a JSJ-piece of $S^3 \backslash \text{int}(N(K))$. Then $M$ has one of the following forms:

\begin{enumerate}

\item an annulus based Seifert fibered space with one singular fiber; 
we call this space a \emph{cable space};
when $M$ contains $\partial E(K)$, the knot $K$ is a cable knot;

\item a Seifert fibered space $S(0,n +1;)$ with $n \ge 2$;  
this space is the complement of the `key-chain' link $H_n$ (Figure \ref{key-chain_3}) as described in \cite{budney2005jsj}; 
we call this space a \emph{composing space};
note that a JSJ-piece adjacent to $M$ can not be again a composing space; 
when $M$ contains $\partial E(K)$, the knot $K$ is a composite knot;
and we say that the component of $H_n$ corresponding to $\partial E(K)$ is \emph{distinguished};


\item a hyperbolic manifold which is homeomorphic to the complement of some hyperbolic link $L$ in $S^3$; 
this link becomes a trivial link or the unknot if a particular component is removed; 
we say this component is \emph{distinguished};
when $M$ contains $\partial E(K)$, this \emph{distinguished} component corresponds to $\partial E(K)$;


\item  a torus knot complement in $S^3$; 
in this case, we have $\partial N(K) \not\subset M$.

\end{enumerate}

\end{theorem}

The JSJ-decomposition of the complement of a knot $K$ in $S^3$ has a natural graph structure. The vertices correspond to JSJ-pieces, and the edges correspond to the JSJ-tori along which adjacent JSJ-pieces are glued. By the generalized Jordan Curve Theorem(\cite{MR0348781}, also see \cite{budney2005jsj}), each embedded torus in $S^3$ seperates, and hence this graph is acyclic. In other words, this graph is a tree.


\begin{figure}[ht]
\includegraphics{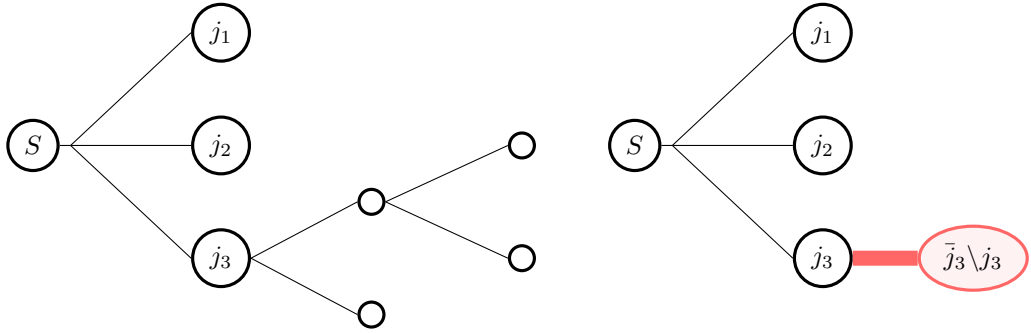}
\caption{\textbf{Left}: The JSJ-tree of a composite knot $J$ with $3$ prime knot summands. The root $S$ is the Seifert fibered space $S(0,3;)$. This composite knot $J$ has three prime summands $J_1, J_2,$ and $J_3$. Each maximal subtree with root labeled by  $j_i$ corresponds to the knot complement $E(J_i)$. In this graph, the prime summands $J_1$ and $J_2$ are non-satellite knots and   their complements have trivial JSJ-decompositions. The prime factor $J_3$ is a satellite knot. Note that the vertex $j_3$ can not be a composing space since $J_3$ is prime by assumption. \textbf{Right}: The same JSJ-tree as the left one, but we use an elliptical node to represent a collection of nodes. The thick edge means that $j_3$ may have multiple subtrees. The notation $\bar{j_3}$ denotes the maximal subtree rooted at $j_3$.
\label{a_typical_JSJ_tree_of_a_knot}}
\end{figure}

\begin{definition}
\label{def:jsj_tree}
We call a graph defined as above a \emph{JSJ-tree}. See Figure \ref{a_typical_JSJ_tree_of_a_knot} for an example.
\end{definition}


Certain unions of JSJ-pieces in $E(K)$ can be regarded as knot complements.

\begin{lemma}[cf. Proof of Theorem 1.1 of \cite{lackenby2017every}]
\label{union_of_JSJ_is_knot_complement}
Let $K$ be a non-trivial knot in $S^3$ and $M$ be a union of JSJ-pieces of $E(K)$. 
Suppose the boundary of $M$ is a torus. 
Then $M$ is the complement of some non-trivial knot in $S^3$. 
\end{lemma}

\begin{proof}
It is well known that each embedded torus in $S^3$ bounds a solid torus. 
This solid torus contains $K$ since $\partial M$ is incompressible in $E(K)$.
By the same argument, 
the other side of this torus, which is $M$, 
is not a solid torus.
Hence $M$ is the complement of a non-trivial knot in $S^3$.
\end{proof}

Now we have described the JSJ-pieces of knot complements.
In order to analyse how these JSJ-pieces fit together, 
we need to parametrize the slopes on their boundary tori.
Our convention is to regard each JSJ-torus as a knot boundary, which is provided by Lemma \ref{union_of_JSJ_is_knot_complement}.

\begin{remark}
\label{non-invertible}
While the orientations of knots do not concern the orientations of their complements,
they are important to connected sums.
However, non-invertible prime summands (of a composite knot $J$) with different orientations do not bother us,
since the slopes at relevant tori are well-defined.
\end{remark}



There is another way to parametrize certain JSJ-tori.
Let $M$ be the JSJ-piece containing $\partial E(K)$.
Suppose $M$ is $S(0,n+1;)$ or a hyperbolic manifold, 
then $M$ is the complement of a key-chain link $H_n$ or of a hyperbolic link. 
We can also use (the components of) this link to parametrize slopes on the boundary tori of $M$ (or of JSJ-pieces adjacent to $M$).

These two parametrizations are related as follows.

 \begin{prop}[\cite{budney2005jsj},\cite{lackenby2017every}]
 \label{parametrizing_slopes_two_methods_relation}
Let $M$ be $S(0,n+1;)$ or a hyperbolic manifold as above. 
Then the slopes defined by the two different parametrizations agree on $\partial E(K)$ 
and are inverse to each other on the remaining boundary components. 
\end{prop}

\begin{proof}

This is a direct consequence of the definition of the `splicing' of \cite{budney2005jsj}. 
See Example $4.9$ of \cite{budney2005jsj}. 
Also see Theorem 4.2 of \cite{lackenby2017every}. 
Note that the torus $\partial E(K)$ corresponds to the distinguished component of the key-chain link $H_n$ or of the hyperbolic link. 
On the other boundary tori, 
the meridians (or longitudes) of the connected components of $E(K)\backslash M$ correspond to the longitudes (or meridians) of the remaining components of $H_n$ or of the hyperbolic link.
\end{proof}

The following two lemmas describe the fiber-slopes of Seifert fibered JSJ-pieces.

\begin{lemma}[Lemma 2.8 of \cite{MR3943698} and cf. Lemma 7.2 of \cite{gordon1983dehn}]
\label{cable_space_fiber_slope}
Suppose $K$ is a cable knot $C_{p',q'}(K')$ in $S^3$. 
then the closure of  $N(K') \backslash N(K)$ is a Seifert fibered space. 
Each of its regular fiber has slope $p'/q'$ at $\partial N(K')$ and $p'q'/1$ at $\partial N(K)$.

\end{lemma}

\begin{lemma}[\cite{budney2005jsj}]
\label{composite_slope}
Let $J$ be a composite knot and $J_i$ be a prime summand.
Suppose $M$ is the JSJ-piece of $E(J)$ containing $\partial E(J)$. 
Then the regular fibers of $M$ has  meridional slopes at $\partial E(J_i)$ and at $\partial E(J)$.

\end{lemma}

\begin{proof}

The Seifert fibered space $S(0,n+1;)$ is the link complement of $H_n$ in $S^3$. The boundary component of $M$ containing $\partial E(J)$  corresponds to the distinguished component of $H_n$. The regular fibers on $\partial E(J)$ are the meridians of this distinguished component, and the regular fibers on $\partial E(J_i)$ are the longitudes of the components of $H_n$ corresponding to $J_i$. Then the statement  follows from Proposition \ref{parametrizing_slopes_two_methods_relation}.
\end{proof}



\section{homeomorphisms of JSJ-pieces}

Let $J$ be a composite knot and $M$ be the JSJ-piece containing $\partial E(J)$.
By Theorem \ref{jsj_of_knot_complement},
$M$ is a composing space $S(0,n + 1;)$.
The fiber-slope of this space is meridional on $\partial E(J)$, 
and the $\alpha/\beta$-slope Dehn filling on $\partial E(J)$ replaces $S(0,n + 1;)$ with $S(0,n; \alpha/\beta)$.




In this section,
we describe the self-homeomorphisms of a Seifert fibered space with boundary.

\begin{definition}
\label{def_annulus_twist}
 We call a self-homeomophism of a $3$-manifold a \emph{vertical Dehn twist} 
 if it is the identity outside a tubular neighborhood of an essential annulus (or of a torus). 
 We call such an annulus (or a torus) a \emph{twist annulus (or a twist torus)}.
\end{definition}

\begin{theorem}[cf. \cite{MR551744}, Section 25]
\label{mapping_class_group_of_sfs}
Let $M$ be a Seifert fibered space with non-empty boundary. Suppose the base orbifold is oriented and the Seifert fibered structure of $M$ is unique. Then the mapping class group of $M$ is generated by vertical Dehn twists and homeomorphisms of the base space which maps singular points to singular points of the same multiplicity.
The twist annuli are essentially those connecting distinct boundary components, and the twists along tori have no effect on the boundary.
\end{theorem}

\begin{remark}
We can assume the orientation-preserving homeomorphisms of the base orbifold restrict to the identity on each boundary component, up to permutations of these boundary circles. This is because the group Homeo$^+(S^1)$ deformation retracts to the group of rotations of $S^1$. 
\end{remark}

\begin{prop}[cf. Proposition 2.1 of \cite{hatcher2000notes}]
\label{twist_cor}
The Seifert fibered spaces $S(0,n;\alpha_1/\beta_1)$ and $S(0,n;\alpha_2/\beta_2)$ ($n \ge 1)$ are homeomorphic as oriented manifolds iff $\alpha_1/\beta_1 \equiv \alpha_2/\beta_2 \pmod{1}$. In addition, any orientation-preserving homeomorphism has the effect of a sequence of non-trivial Dehn twists along the fiber direction on at least one boundary component.
\end{prop}

\begin{proof}
An homeomorphism from $S(0,n;\alpha_1/\beta_1)$ to  $S(0,n;\alpha_2/\beta_2)$ is induced by vertical Dehn twists such that at least one twist annulus connects the tubular neighborhood of the fiber of type $\alpha_1/\beta_1$ and a boundary component of the Seifert fibered space $S(0,n;\alpha_1/\beta_1)$. (Strictly speaking, these particular vertical Dehn twists are defined for $S(0,n;)$ and the homeomorphism extends to a map from $S_1$ to $S_2$.) In particular, this homeomorphism restricts to a Dehn twist along a regular fiber on some boundary torus. Since any orientation-preserving homeomorphism is the map above composed with an automorphism of $S(0,n;\alpha_2/\beta_2)$, the second statement follows from Theorem \ref{mapping_class_group_of_sfs}.
\end{proof}

When dealing with homeomorphisms from one manifold to another, we need to know how the slopes on some JSJ-tori are mapped. 
When the slopes are parametrized by knot complements bounded by the relevant JSJ-torus, 
orientation-preserving homeomorphisms preserve the slopes due to the following theorem.

\begin{theorem}[The Knot Complement Theorem, \cite{gordon1989knots}]

\label{gordon_luecke}
Two knots are isotopic if and only if their complements in $S^3$ are homeomorphic as oriented manifolds. In addition, each such homeomorphism sends meridians to meridians and longitudes to longitudes.

\end{theorem}

\section{proof of the main theorem}
In this section, 
we prove that $S^3_r(J) \not\cong S^3_{s}(J)$ for $r\neq s$.
First, we exclude most cases by the following proposition.
\begin{prop}
\label{first_reduction}
If $S^3_r(J) \cong S^3_{s}(J)$, 
then the pair $\{r, s\}$ is either $\{\pm 1\}$, $\{\pm2\}$, or $\{\pm 1/2\}$.
\end{prop}


\begin{proof}
By Theorem \ref{hanselmans theorem},
we have $r \in \{\pm 1, \pm2,  \pm 1/q\}$ with $q>1$.
When $r = 1/q$ with $q > 2$,
the two Dehn filled Seifert fibered spaces $S_1$ and $S_{2}$ are not homeomorphic as oriented manifolds,
by Proposition \ref{twist_cor}.
Then the statement follows from Lemma \ref{diff_sfs_piece}. 
\end{proof}


For convenience, we introduce the following notation.



\begin{notation}
Given $n$ and $r = \alpha/\beta$, 
we use $S_0$, $S_1$ and $S_{2}$ to denote 
the Seifert fibered JSJ-pieces
$S(0,n + 1;) \subset E(J)$, $S(0,n; \alpha/\beta)\subset S^3_r(J)$, 
and $S(0,n; -\alpha/\beta) \subset S^3_{-r}(J)$ respectively.
\end{notation}




For the rest of this article,
we assume the existence of a potential homeomorphism $h : S^3_r(J) \to S^3_{-r}(J)$.
This map should send $S_1$ to $S_2$ or elsewhere.
We show that both cases lead to contradictions,
and hence such an $h$ does not exist.
The following lemma tells us $h(S_1) \neq S_2$.
Note that we can naturally identify $S^3_r(J)\backslash S_1$ with $E(J)\backslash S_0$.

\begin{lemma}
\label{not_fixing_root} 
There is no orientation-preserving homeomorphism 
from $S^3_r(J)$ to $S^3_{-r}(J)$ 
which sends $S_1$ to $S_2$.
\end{lemma}
 
\begin{proof}

We first suppose that $h$ fixes each connected component of  $E(J)\backslash S_0$. 
By Theorem \ref{gordon_luecke}, 
the map $h$ preserves the slopes of each connected component of $\partial E(J)\backslash S_0$, 
i.e., 
it preserves the slopes on each component of $\partial S_1$. 
This contradicts the fact that $h$ has the effect of non-trivial Dehn twists on at least one boundary component of $S_1$, by Proposition \ref{twist_cor}. 

Now suppose that some of the prime summands of $J$ are isotopic, 
probably with different orientations. 
(cf. Remark \ref{non-invertible}.)
The map $h$ is slope-preserving on the boundary components of $S_1$ as above,
which again contradicts Proposition \ref{twist_cor}.   
\end{proof}




From now on, 
we assume $h(S_1) \neq S_2$,
and we only consider pairs $\{\pm 1\}$, $\{\pm 2\}$ and $\{\pm 1/2\}$.

\begin{cor}
\label{prime_factors_greater_than_two}
Let $J$ be a composite knot.
Suppose it has at least three prime summands and $r = \pm 1/2$, or it has two prime summands and $r \in \mathbb{Z}$.
Then $S^3_r(J) \not\cong S^3_{-r}(J).$
\end{cor}

\begin{proof}
If $n \ge 3$ and $r = \pm 1/2$, then $S_1$ is a Seifert fibered space with $\ge 3$ boundary components and with a singular fiber; if $n = 2$ and $r \in \mathbb{Z}$, then $S_1$ is a Seifert fibered space with two boundary components and without singular fibers. 
In both cases, there is no JSJ-piece of $S^3_{-r}(J) \backslash S_2$ homeomorphic to $S_1$, by Theorem \ref{jsj_of_knot_complement}. 
Thus any homeomorphism from $S^3_r(J)$ to $S^3_{-r}(J)$ must send $S_1$ to $S_2$, 
then the statement follows from Lemma \ref{not_fixing_root}.
\end{proof}


\begin{notation}
Let $J_i$ be a prime summand of $J$.
We use $\bar{Z}_i$ to denote the part of $S^3_r(J)$
corresponding to $E(J_i)$ in the JSJ-tree.
For that of $S^3_{-r}(J)$,
we use the notation $\bar{Z}_i'$.
Using these notations, 
each JSJ-tree is simply the union of $S_1$ (or $S_2$) and some $\bar{Z}_i$'s (or $\bar{Z}_i'$'s).
Since we assume $h(S_1) \neq S_2$,
it follows that $h(S_1)$ must be contained in some $\bar{Z}_i'$.
We use  $\bar{Y}'$ instead of $\bar{Z}_i'$,
indicating that it contains $h(S_1)$,
and use $J_Y$ to denote the corresponding prime summand. 
We also use a set of notations without `bar'.
For example, $Y'$ means the JSJ-piece in $\bar{Y}'$ containing $\partial \bar{Y}'$.
See Figure \ref{larger_than_3}.
\end{notation}

\begin{figure}[ht]
\includegraphics{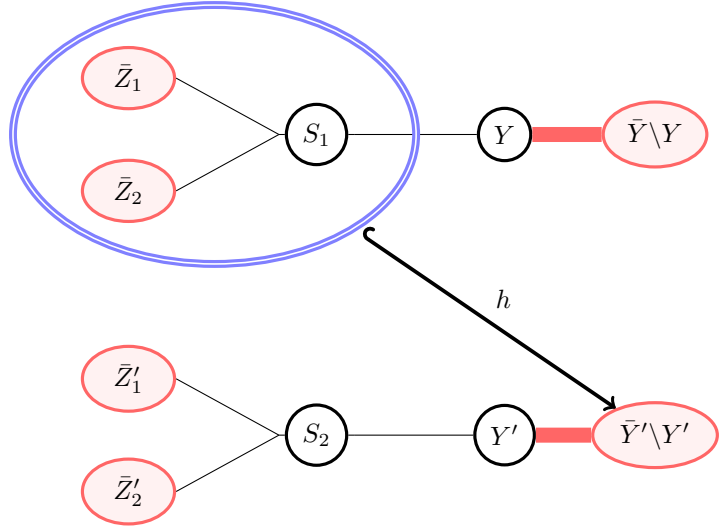}
\caption{This picture illustrates the JSJ-trees of $S^3_r{(J)}$ (above) and $S^3_{-r}(J)$ (below) when $n = 3$. For $n > 3$, we just have more $\bar{Z_i}$ nodes. Each round node represents a JSJ-piece, and each elliptical node represents a union of JSJ-pieces. The elliptical nodes $\bar{Z}_i$ and $\bar{Z}_i'$ denote the knot complements $E(J_i)$. The thick lines between $Y$ and $\bar{Y}\backslash Y$ denotes a multiple-edge, meaning that $Y$ may have more than one subtrees. The nodes contained in the large ellipse forms $\bar{Z}\cup S_1$, which is mapped into $\bar{Y}'$ by Lemma \ref{map_to_Ybar}. In the case $n \ge 3$, we have $h(S_1) \neq Y'$ and hence $\bar{Z}\cup S_1$ is mapped into $\bar{Y}'\backslash Y'$ (Lemma \ref{map_to_Ybar}).
\label{larger_than_3}}
\end{figure}

\begin{lemma}
\label{map_to_Ybar}
We have $h(\bar{Z}_i \cup S_1) \subset \bar{Y}'$ for each $i$. 
Furthurmore, 
we have $h(\bar{Z}_i \cup S_1) \subset \bar{Y}'\backslash Y'$ when $n \ge 3$. 
\end{lemma}

\begin{proof}
Since $h(S_1) \subset \bar{Y}'$, 
there is a connected component of $S^3_{-r}(J)\backslash h(S_1)$ 
which contains $S_2$ and hence contains each $\bar{Z_i}'$. 
This component is actually $h(\bar{Y})$ since it contains more JSJ-pieces than each $\bar{Z_i}$ does. 
Indeed, the other connected components (which are not shown in Figure \ref{larger_than_3}) are contained in $\bar{Y}'$. 
Thus each $h(\bar{Z_i})$ corresponds to a connected component of $S^3_{-r}(J)\backslash h(S_1)$ contained in $\bar{Y}'$.  

When $n \ge 3$, the spaces $S_2$ and $ h(S_1)$ are composing spaces with $\ge 3$ boundary components (cf. Corollary \ref{prime_factors_greater_than_two}).
Since $Y'$ is adjacent to $S_2$, it can not be a composing space,
by Theorem \ref{jsj_of_knot_complement}. 
Hence we have $h(S_1) \neq Y'$.
\end{proof}

\begin{lemma}
\label{some_knot_complements}
The spaces $h(\bar{Z} \cup S_1)$ and $h(Y\cup S_1 \cup \bar{Z})$
can be regarded as knot complements.
\end{lemma}

\begin{proof}
By Lemma \ref{map_to_Ybar}, both spaces are contained in $\bar{Y}'$,
and hence can be regarded as submanifolds in $S^3$.
Then the statement follows from Lemma \ref{union_of_JSJ_is_knot_complement}.
\end{proof}

\begin{prop}
\label{two_factors_r_1/2}
Suppose $J$ has two prime summands.
Then  $S^3_{1/2}(J) \not\cong S^3_{-1/2}(J)$.
\end{prop}

\begin{proof}
Since $J$ has two prime summands,
$S_1$ is a Seifert fibered space with two boundary components
and with one singular fiber,
i.e.,
a cable space.
This also means that $h(S_1)$ is a cable space.
The image $h(\bar{Z}_1 \cup S_1)$ and $h(\bar{Z}_1)$ can be regarded as  knot complements in $S^3$, 
by Lemma \ref{union_of_JSJ_is_knot_complement} and Lemma \ref{map_to_Ybar}. 
Furthermore, the space $h(\bar{Z}_1 \cup S_1)$ must be the complement of a cable of a non-trivial knot since $h(S_1)$ is a cable space. 
By Lemma \ref{cable_space_fiber_slope}, the fiber-slope of $h(S_1)$ at $\partial h(\bar{Z}_1)$ is not meridional. 
On the other hand, the fiber-slope of $S_1$ at $\partial \bar{Z}_1$ is meridional. 
Since $h$ preserves slopes and sends fibers to fibers, we have a contradiction. 
\end{proof}

It remains to deal with the case when $n \ge 3$ and $r = \pm 1$ or $\pm2$.
However, the slopes $r = \pm 2$ can be excluded by Theorem \ref{hanselmans theorem},
since the genus of $J$ is greater than two when $n \ge 3$.
Thus the difficulty lies in dealing with the slopes $\pm 1$.

We denote  the disjoint union of all the $\bar{Z}_i$'s (not including $\bar{Y}'$) by $\bar{Z}$.

\begin{lemma}
\label{meridians_are_glued_together}
The space $\bar{Z} \cup S_1$  can be regarded as a  knot complement. 
The fiber-slope of $S_1$ at the boundary of $\bar{Z} \cup S_1$ is meridional. 
In particular, the meridians of $\bar{Y}$ are glued to the meridians of $\bar{Z} \cup S_1$.  

\end{lemma}

\begin{proof}


The first statement follows from Lemma \ref{some_knot_complements}.
An example of $\bar{Z} \cup S_1$ is circled in Figure \ref{larger_than_3} by a large ellipse.
For the second statement, note that  $S_1$ is a composing space with $\ge 3$ boundary components.
Then $\bar{Z} \cup S_1$ can be regarded as the  complement of some composite knot $J'$,
and the fiber-slope of $S_1$ is meridional at $\partial E(J')$ parametrized using $\bar{Z} \cup S_1 $,
by Lemma \ref{composite_slope}.
The fiber-slope of $S_1$ is again meridional at $\partial \bar{Y}$, parametrized using $\bar{Y}$.
Thus we obtain the last statement.
\end{proof}


\begin{figure}[ht]
\includegraphics{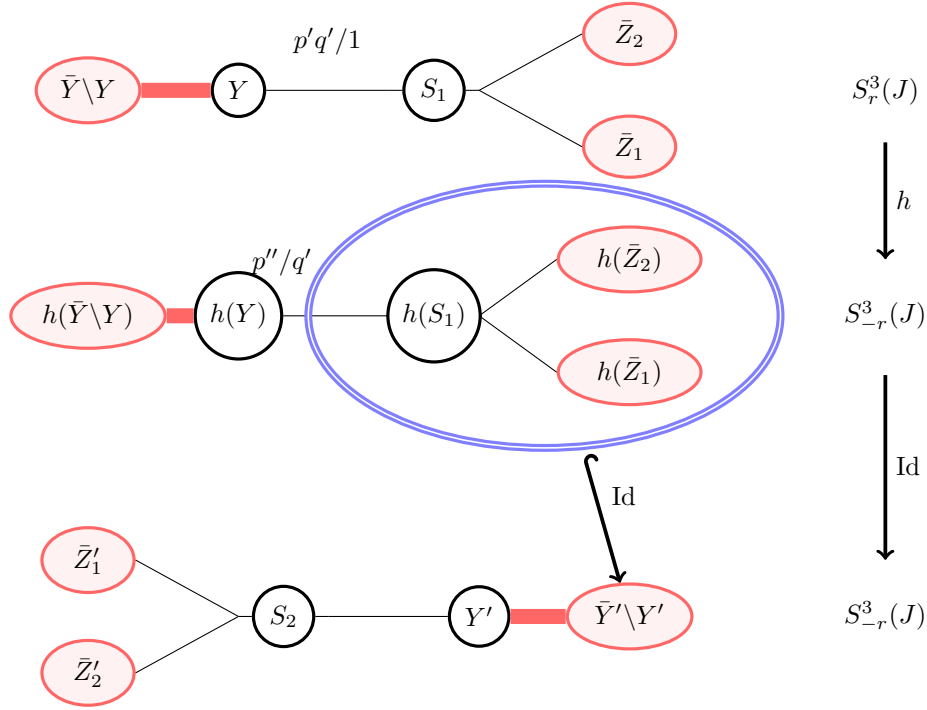}
\caption{ \textbf{Illustration of Proposition \ref{three_factors_r_Z}.} The first row is the JSJ-tree for $S^3_r(J)$ and the other two rows are the JSJ-trees for $S^3_{-r}(J)$. These two JSJ-trees for $S^3_{-r}(J)$ look different because different parts of the JSJ-trees are explicitly drawn. (Recall that each elliptical node represents a collection of JSJ-pieces.) The fiber-slopes are added for the case that $Y$ is a cable space. The case of hyperbolic $Y$:  in the first row, the spaces $\bar{Y}$ and $S_1\cup \bar{Z}$ are glued by identifying meridians with meridians; (both spaces are knot complements;) in the second row, the corresponding spaces are glued by identifying longitudes with meridians.
\label{Y_is_cable_or_hyperbolic}}
\end{figure}

\begin{prop}
\label{three_factors_r_Z}
$S^3_1(J) \not\cong S^3_{-1}(J)$.
\end{prop}

\begin{proof}
First, we assume $Y$ is a cable space. 
It follows that $h(Y)$ is also a cable space with a singular fiber of the same multiplicity. 
By Theorem \ref{jsj_of_knot_complement}, 
the space $h(\bar{Y})$ is the complement of some $(p',q')$-cable, 
and  $h(Y\cup S_1 \cup \bar{Z})$ is the complement of some $(p'',q')$-cable.
(cf. Lemma \ref{some_knot_complements}.) 
By Lemma \ref{cable_space_fiber_slope}, 
the fiber-slope of $Y$ at $\partial  \bar{Y}$ is $p'q'/1$, 
parametrized by $\bar{Y}$,
and the fiber-slope of $h(Y)$ at $\partial h(\bar{Z}\cup S_1)$ is $p''/q'$,
parametrized by $h(Y\cup S_1 \cup \bar{Z})$. 
(See Figure \ref{Y_is_cable_or_hyperbolic}.) 
The intersection number of the fibers of $Y$ with the meridians of $\bar{Y}$ is one. 
By Lemma \ref{meridians_are_glued_together}, 
the fibers also intersect once with the meridians of $\bar{Z}\cup S_1$. 
However, 
the intersection number of the fibers of $h(Y)$ with meridians of $h(\bar{Z}\cup S_1)$ is $q' (\ge 2)$. 
Since $h$ preserves both slopes and fibers, we have a contradiction.

Second, we assume $Y$ is the complement of a hyperbolic link $L$. 
Regarding $\bar{Y}$ as the knot complement $E(J_Y)$,
its boundary corresponds to the distinguished component $L_0$ of $L$,
and its longitudes are identified with the longitudes of $L_0$.
Since $h$ preserves slopes,
the longitudes of $h(\bar{Y})$ should be identified with the longitudes of $h(L_0)$.
Now recall that $h(Y\cup S_1 \cup \bar{Z})$ can be regarded as a knot complement.
By Proposition \ref{parametrizing_slopes_two_methods_relation},
the meridians of $h(\bar{Z}\cup S_1)$ are the longitudes of $h(L_0)$.
(Note that $h(L_0)$ is not the distinguished component of $h(L)$.)
Hence we see that the longitudes of $h(\bar{Y})$ are identified with the meridians of $h(\bar{Z}\cup S_1)$.
However, the meridians of $\bar{Y}$ are glued to the meridians of $\bar{Z}\cup S_1$,
by Lemma \ref{meridians_are_glued_together}. 
Since $h$ preserves slopes by Theorem \ref{gordon_luecke}, 
we have a contradiction.
\end{proof}



Combining  Proposition \ref{first_reduction},
\ref{two_factors_r_1/2},
\ref{three_factors_r_Z},
and 
Corollary \ref{prime_factors_greater_than_two},
we have the following theorem.

\begin{theorem}
Let $J$ be a composite knot. 
Suppose there exists an orientation-preserving homeomorphism $h: S^3_{r}(J) \to S^3_{s}(J)$, 
then $r = s$.
\end{theorem}

\bibliographystyle{ha}
\bibliography{composite.bib}

\end{document}